\documentclass[a4paper,notitlepage,twoside,reqno,10pt]{amsart}

\usepackage{bbm,pifont,latexsym}

\usepackage{dcolumn,indentfirst}
\usepackage[bookmarksopen=true,linktocpage=true,pdfstartview={XYZ null null 1.5}]{hyperref}
\usepackage{amsmath,amssymb,amscd,amsthm,amsfonts,mathrsfs,yhmath}
\usepackage{color,graphicx,xcolor,graphics,subfigure,extarrows,caption2}
\usepackage{titletoc}

\newtheorem{thm}{Theorem}[section]
\newtheorem{lema}[thm]{Lemma}

\newtheorem{thmx}{Theorem}

\theoremstyle{definition}
\newtheorem*{defi}{Definition}

\theoremstyle{remark}

\newtheorem*{ques}{\hskip 1em Question}

\newcommand{\D}{\mathbb{D}}

\newcommand{\R}{\mathbb{R}}
\newcommand{\Z}{\mathbb{Z}}
\newcommand{\N}{\mathbb{N}}
\newcommand{\Q}{\mathbb{Q}}

\newcommand{\C}{\mathbb{C}}
\newcommand{\EC}{\widehat{\mathbb{C}}}

\newcommand{\MM}{\mathcal{M}}

\newcommand{\MH}{\mathcal{H}}

\newcommand{\MS}{\mathcal{S}}

\newcommand{\Crit}{\textup{Crit}}

\newcommand{\diam}{\textup{diam}}

\newcommand{\ii}{\textup{i}}

\newcommand{\Area}{\textup{Area}}
\newcommand{\HT}{\textup{HT}}

\begin{document}

\author{Yuming Fu}
\address{Department of Mathematics, Nanjing University, Nanjing 210093, P. R. China}
\email{476876690@qq.com}

\author{FEI YANG}
\address{Department of Mathematics, Nanjing University, Nanjing 210093, P. R. China}
\email{yangfei@nju.edu.cn}

\title[Area and Hausdorff dimension of carpet Julia sets]{Area and Hausdorff dimension of Sierpi\'{n}ski carpet Julia sets}

\begin{abstract}
We prove the existence of rational maps whose Julia sets are Sierpi\'{n}ski carpets having positive area. Such rational maps can be constructed such that they either contain a Cremer fixed point, a Siegel disk or are infinitely renormalizable. We also construct some Sierpi\'{n}ski carpet Julia sets with zero area but with Hausdorff dimension two. Moreover, for any given number $s\in(1,2)$, we prove the existence of Sierpi\'{n}ski carpet Julia sets having Hausdorff dimension exactly $s$.
\end{abstract}



\date{\today}



\maketitle

\section{Introduction}\label{introduction}

According to \cite{Wh58}, a subset $S$ in the Riemann sphere $\EC$ is called a \emph{Sierpi\'{n}ski carpet} (\emph{carpet} in short) if $S$ is compact, connected, locally connected, has empty interior and has the property that the complementary domains are bounded by pairwise disjoint simple closed curves. All Sierpi\'{n}ski carpets are homeomorphic. In particular, they are all homeomorphic to the standard ``middle ninths" square Sierpi\'{n}ski carpet, which is the well-known self-similar set shown in Figure \ref{Fig:carpet}.

\begin{figure}[!htpb]
  \setlength{\unitlength}{1mm}
  \centering
  \includegraphics[width=0.45\textwidth]{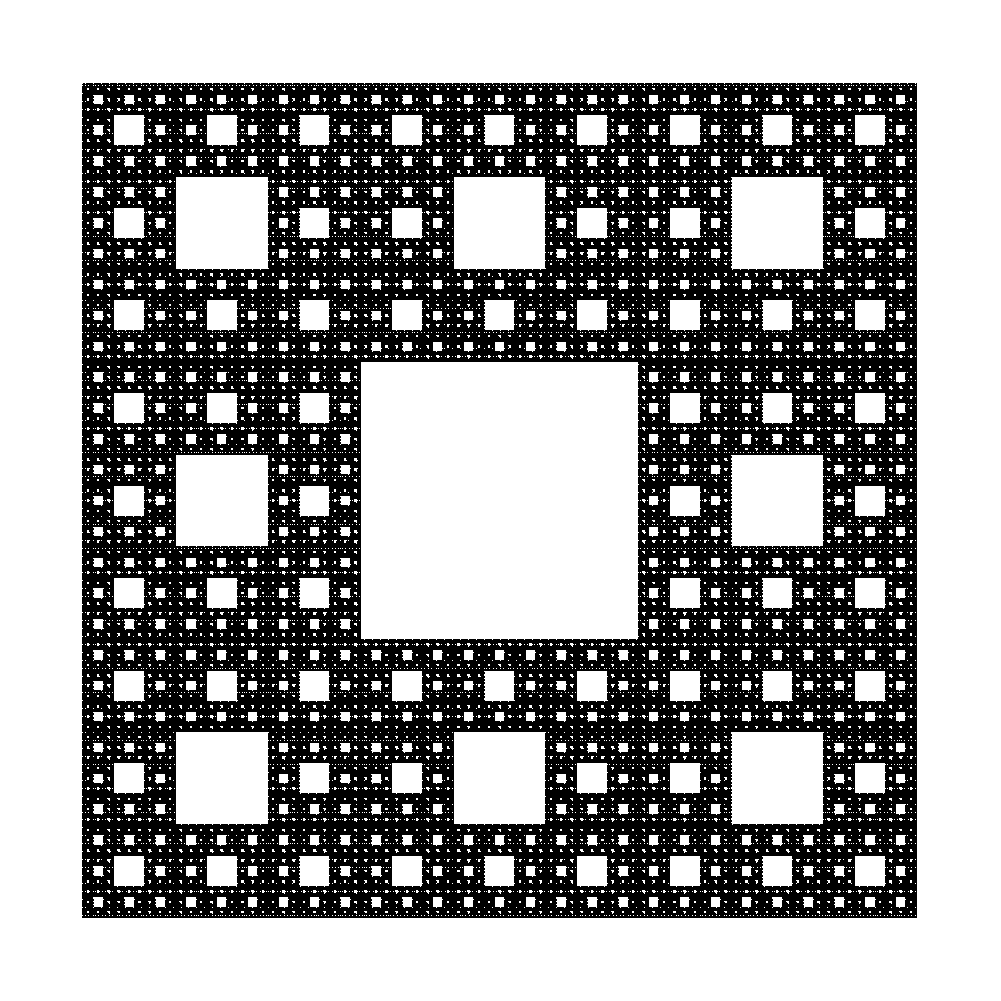}~~
  \includegraphics[width=0.45\textwidth]{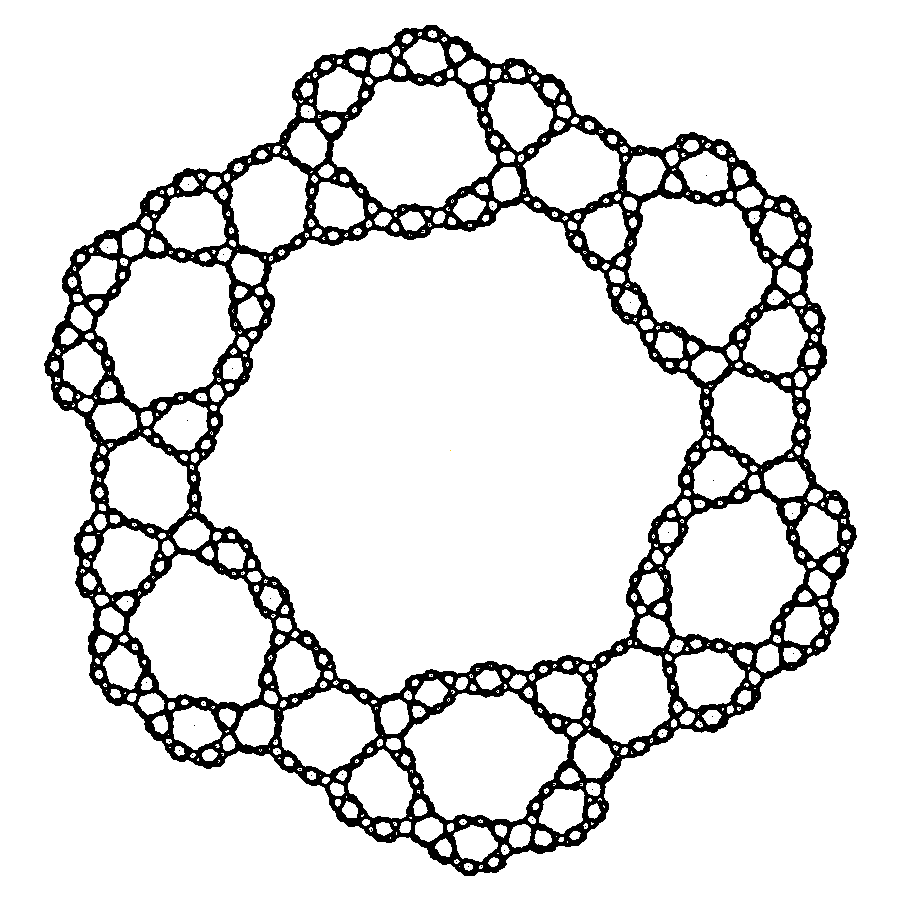}
  \caption{Left: The middle ninths square Sierpi\'{n}ski carpet. Right: Sierpi\'{n}ski carpet Julia set of a hyperbolic McMullen map. They are homeomorphic to each other and both have Hausdorff dimension strictly less than two.}
  \label{Fig:carpet}
\end{figure}

Since all Sierpi\'{n}ski carpets are the same in the sense of topological equivalence, many people are interested on the geometric classifications of Sierpi\'{n}ski carpets. In particular, the study of quasisymmetric equivalences between general carpets and round carpets (i.e. the boundaries of complement components of carpets are round circles) is partially motivated by the Kapovich-Kleiner conjecture in geometric group theory. The quasisymmetric rigidity of the square Sierpi\'{n}ski carpets was also studied extensively recently. One may refer to \cite{Bo11}, \cite{BM13}, \cite{Me14}, \cite{ZS15} and the references therein.

As chaotic sets, the Sierpi\'{n}ski carpets also appear as Julia sets of rational maps and they play an important role in complex dynamics \cite{Mc94a}.
The first example of Sierpi\'{n}ski carpet Julia sets was found by Milnor and Tan Lei \cite[Appendix F]{Mi93}. They constructed such Julia sets in quadratic rational maps having two super-attracting cycles with periods $3$ and $4$. Many Sierpi\'{n}ski carpet Julia sets was found later in quadratic rational maps containing super-attracting cycles with various periods \cite{DFGJ14}. Another important family, the McMullen maps
\begin{equation}\label{equ:McM-1}
F_\lambda(z)=z^n+\frac{\lambda}{z^m}, \text{\quad where } n\geq 2,\,m\geq 1 \text{ and } \lambda\in\C\setminus\{0\},
\end{equation}
can produce a lot of ``different" Sierpi\'{n}ski carpet Julia sets (see \cite{DLU05}, \cite{DP09} and Figure \ref{Fig:carpet}). One may also refer to \cite[\S 5.6]{Pi94}, \cite{Mo00}, \cite{St08}, \cite{Lo10}, \cite{XQY14} and \cite{Ya18} for some other rational maps having Sierpi\'{n}ski carpet Julia sets.

The quasisymmetric equivalence and rigidity on Sierpi\'{n}ski carpet Julia sets were studied initially in \cite{BLM16} for critically finite rational maps, and then some corresponding results were generalized to the critically infinite case recently \cite{QYZ19}. Moreover, some topological characterizations of Sierpi\'{n}ski carpet Julia sets of critically finite rational maps are given in \cite{GZZ17} and \cite{GHMZ18}.

\vskip0.1cm
To study the quasisymmetric rigidity of the Sierpi\'{n}ski carpet Julia sets, it is important to know if these Julia sets have zero area. Except for an infinitely renormalizable example given in \cite{QYY18}, to the best of our knowledge, all the known Sierpi\'{n}ski carpet Julia sets are hyperbolic, parabolic or semi-hyperbolic and hence they all have Hausdorff dimension strictly less than two. In this paper we prove the following result.

\begin{thmx}\label{thm-area}
There exist rational maps that have Sierpi\'{n}ski carpet Julia sets with positive area. Moreover, such rational maps can be constructed so that they either contain a Cremer fixed point, a Siegel disk or are infinitely renormalizable.
\end{thmx}

The existence of quadratic polynomials that have Julia sets with positive area is known (see \cite{BC12}, \cite{AL15}). Such quadratic polynomials either contain a Cremer fixed point, a Siegel disk or are infinitely renormalizable. We will use polynomial-like mapping and renormalization theory to ``arrange" these quadratic Julia sets in the Julia sets of rational maps such that they are Sierpi\'{n}ski carpets. We would like to mention that these are the first examples of Sierpi\'{n}ski carpet Julia sets containing the Cremer points and the boundaries of Siegel disks.

\vskip0.1cm
Using the same idea (polynomial-like mapping and renormalization theory) we also prove the following result.

\begin{thmx}\label{thm-area-0-dim-2}
There exist rational maps that have Sierpi\'{n}ski carpet Julia sets with zero area and Hausdorff dimension two.
\end{thmx}

It was known from Shishikura and Lyubich in 1990s that there exist quadratic polynomials that have Julia sets with full Hausdorff dimension but with zero area. One of the crucial points in the proof of Theorem \ref{thm-area-0-dim-2} is to show that the whole Julia sets of the rational maps have zero area. The rational maps constructed in Theorems \ref{thm-area} and \ref{thm-area-0-dim-2} are both McMullen maps defined in \eqref{equ:McM-1}.

\begin{thmx}\label{thm-dim}
For each $s\in(1,2)$, there exist rational maps that have Sierpi\'{n}ski carpet Julia sets with Hausdorff dimension exactly $s$.
\end{thmx}

Ha\"{i}ssinsky and Pilgrim proved that there are Sierpi\'{n}ski carpet Julia sets with conformal dimensions (and hence Hausdorff dimensions) arbitrarily close to two \cite{HP12}.
Note that their result and Theorem \ref{thm-dim} are not related to each other. It is natural to ask that if the result of Theorem \ref{thm-dim} is sharp, i.e.,
\begin{ques}
Do the Sierpi\'{n}ski carpet Julia sets always have Hausdorff dimension strictly larger than one?
\end{ques}

We will prove that if a rational map has an attracting or parabolic basin, then the corresponding Julia set has Hausdorff dimension strictly larger than one provided it is a Sierpi\'{n}ski carpet (see Theorem \ref{thm:att-para}). If the Sierpi\'{n}ski carpet Julia set only contains the boundaries of Siegel disks and their preimages, we can show that the answer to the above question is yes in some special cases (see \S\ref{subsec:dim-g-1}). However, we can prove the following

\begin{thmx}\label{thm:dim-1}
There exist Sierpi\'{n}ski carpets (not Julia sets) with Hausdorff dimension one.
\end{thmx}

The examples in Theorem \ref{thm:dim-1} open the possibility of the existence of  Sierpi\'{n}ski carpet Julia sets with Hausdorff dimension one. However, from \S\ref{subsec:dim-1}, the Sierpi\'{n}ski carpets constructed in Theorem \ref{thm:dim-1} are far from self-similar. On the other hand, Julia sets of rational maps have self-similarity in general (even if they contain the boundaries of the Siegel disks). This observation indicates that the answer to the above question should be affirmative.

\vskip0.2cm
\textit{Acknowledgements.} The authors are very grateful to Huojun Ruan for providing a method to construct the Sierpi\'{n}ski carpets (not Julia sets) with Hausdorff dimension one (Theorem \ref{thm:dim-1}), to Xiaoguang Wang, Yongcheng Yin and Jinsong Zeng for offering a proof of Lemma \ref{lema:renor-attr}. We would also like to thank Arnaud Ch\'{e}ritat, Kevin Pilgrim, Feliks Przytycki, Weiyuan Qiu, Yongcheng Yin and Anna Zdunik for helpful discussions and comments. This work is supported by National Natural Science Foundation of China.

\section{Definitions and basic settings}\label{sec:preliminary}

In this section we first give some necessary definitions in polynomial-like renormalization theory, and then state some useful results on the dynamics of McMullen maps and  quadratic polynomials.

\subsection{Polynomial-like mappings and renormalization}

Let $U$ and $V$ be two Jordan disks in $\C$ such that $U$ is compactly contained in $V$. The triple $(f,U,V)$ is called a \textit{polynomial-like mapping} of degree $d\geq 2$ if $f:U\to V$ is a proper holomorphic surjection with degree $d$. The \emph{filled Julia set} $K(f)$ is defined by $K(f)=\bigcap_{n\in\N}f^{-n}(V)$ and the \textit{Julia set} $J(f)$ is the topological boundary of $K(f)$.

Two polynomial-like mappings $(f_1,U_1,V_1)$ and $(f_2,U_2,V_2)$ of degree $d$ are said to be \emph{hybrid equivalent} if there exists a quasiconformal homeomorphism $h$ defined from a neighborhood of $K(f_1)$ onto a neighborhood of $K(f_2)$, which conjugates $f_1$ to $f_2$ and the complex dilatation of $h$ on $K(f_1)$ is zero. The following result is fundamental in the renormalization theory, which is due to Douady and Hubbard.

\begin{thm}[{The Straightening Theorem, \cite[p.\,296]{DH85}}]\label{straightening}
Let $(f,U,V)$ be a polynomial-like mapping of degree $d\geq 2$. Then
\begin{enumerate}
\item $(f,U,V)$ is hybrid equivalent to a polynomial $P$ with the same degree $d$;
\item If $K(f)$ is connected, then $P$ is uniquely determined up to a conjugation by an affine map.
\end{enumerate}
\end{thm}

In this paper we are only interested in \textit{quadratic-like mappings}, that are the polynomial-like mappings with degree $d=2$. According to Theorem \ref{straightening}, every quadratic-like mapping $(f,U,V)$ is hybrid equivalent to a quadratic polynomial
\begin{equation*}
Q_c(z)=z^2+c \text{\quad where\quad}c\in\C.
\end{equation*}
If the unique critical orbit of $f$ is contained in $U$, then $Q_c$ is unique. We use $\beta_c$ to denote the \textit{$\beta$-fixed point} (i.e., the landing point of the zero external ray) of $Q_c$ and $\beta_c'$ the other preimage of $\beta_c$.

\vskip0.1cm
Let $f$ be a rational map. The \textit{post-critical set} $P(f)$ of $f$ is defined as the closure of $\bigcup_{n\geq 1}f^{\circ n}(\Crit(f))$, where $\Crit(f)=\{c:f'(c)=0\}$ is the set of critical points of $f$. If there exist $p\geq 1$ and two Jordan domains $U\Subset V$ such that $(f^{\circ p},U,V)$ is a quadratic-like mapping with connected Julia set\footnote{The definition of renormalization needs to exclude a special case, i.e. $f$ itself is a polynomial and $U$ is an open neighborhood of the filled Julia set of $f$.}, then $f$ is called \textit{$p$-renormalizable}. The sets $K(f)$, $f(K(f))$, $\cdots$, $f^{\circ (p-1)}(K(f))$ are called \textit{small filled Julia sets}, and $J(f)$, $f(J(f))$, $\cdots$, $f^{\circ (p-1)}(J(f))$ are called \textit{small Julia sets}. The $\beta$-fixed point $\beta_{f^{\circ p}}$ of $(f^{\circ p},U,V)$ is defined as $h^{-1}(\beta_c)$ and the other preimage of $\beta_{f^{\circ p}}$ is defined as $\beta_{f^{\circ p}}'=h^{-1}(\beta_c')$, where $c\in\C$ is uniquely determined such that $f^{\circ p}$ is hybrid conjugated to $Q_c$ by $h$.

\begin{defi}[{Infinitely renormalization}]
The renormlization is \textit{primitive} type if the small Julia sets $J(f)$, $f(J(f))$, $\cdots$, $f^{\circ (p-1)}(J(f))$ are pairwise disjoint. Otherwise, the renormlization is \textit{satellite} type. The map $f$ is \textit{infinitely renormalizable} if $f$ is $p$-renormalizable for infinitely many positive integers $p$.
\end{defi}

For more details on the backgrounds and results on polynomial-like renormlization, see \cite{Mc94b}.

\subsection{Dynamics of McMullen maps}

In this subsection we present some basic notations and results of McMullen maps
\begin{equation}\label{equ:McM-2}
f_\lambda(z)=z^n+\frac{\lambda}{z^n}, \text{\quad where } n\geq 3 \text{ and } \lambda\in\C^*=\C\setminus\{0\}.
\end{equation}
The map $f_\lambda$ has a super-attracting basin of infinity. We use $B_\lambda$ to denote the immediate super-attracting basin of infinity. Note that $f_\lambda^{-1}(\infty)=\{\infty,0\}$. We denote by $T_\lambda$ the Fatou component (i.e. the \textit{trap door}) of $f_\lambda$ containing the origin. It is easy to check that $f_\lambda$ has $2n$ ``free" critical points $c_0$, $c_1$, $\cdots$, $c_{2n-1}$ satisfying $c_i^{2n}=\lambda$, where $0\leq i\leq 2n-1$, and $f_\lambda$ has exactly two ``free" critical values $v_\lambda^\pm=\pm\,2\sqrt{\lambda}$. The following theorem implies that the dynamics of $f_\lambda$ is determined by any one of the free critical orbits.

\begin{thm}[\cite{DLU05}, \cite{DR13}]\label{thm:Escape-3}
Suppose that any one of the free critical points of $f_\lambda$ is attracted by $\infty$. Then one and only one of the following three cases happens:
\begin{enumerate}
\item $f_\lambda(c_i)\in B_\lambda$ for some $i$, then $J(f_\lambda)$ is a Cantor set;
\item $f_\lambda(c_i)\in T_\lambda\neq B_\lambda$ for some $i$, then $J(f_\lambda)$ is a Cantor set of circles;
\item $f_\lambda^{\circ (k-1)}(c_i)\in T_\lambda\neq B_\lambda$ for some $i$ and $k\geq 3$, then $J(f_\lambda)$ is a Sierpi\'{n}ski carpet.
\end{enumerate}
Moreover, if the forward orbit any one of the free critical points of $f_\lambda$ is bounded, then $J(f_\lambda)$ is connected.
\end{thm}

For the topology of the immediate super-attracting basin of infinity, Devaney conjectured that $\partial B_\lambda$ is always a Jordan curve (if it is not a Cantor set) and this has been proved by Qiu, Wang and Yin.

\begin{thm}[{\cite[Theorem 1.1]{QWY12}}]\label{thm:QWY}
The boundary of the immediate basin of infinity $\partial B_\lambda$ of $f_\lambda$ is always a Jordan curve provided $J(f_\lambda)$ is not a Cantor set.
\end{thm}

For $k\geq 0$, we denote
\begin{equation*}
\MH_{k}:=\{\lambda \in \C^{*}: k \text{ is the minimal integer such that } f_{\lambda}^{\circ k}(c_0)\in B_{\lambda}\}.
\end{equation*}
Every component of $\MH_{k}$ is called an \emph{escape domain}. In particular, $\MH_0$ is called the \textit{Cantor locus}; $\MH_1=\emptyset$; $\MH_2$ is the \textit{McMullen domain}; and all connected components of $\MH_k$ with $k\geq 3$ are called \textit{Sierpi\'{n}ski holes}. The complement of these escape domains is called the \emph{non-escape locus} $\mathcal{N}$.
Moreover, every hyperbolic component in $\mathcal{N}$ is contained in some homeomorphic image of the Mandelbrot set which corresponds to the renormalizable or $*$-renormalizable parameters (see \cite[\S 7]{St06}, \cite[\S 5]{QWY12} and Figure \ref{Fig:McM-parameter}).

\begin{figure}[!htpb]
  \setlength{\unitlength}{1mm}
  \centering
  \includegraphics[width=0.45\textwidth]{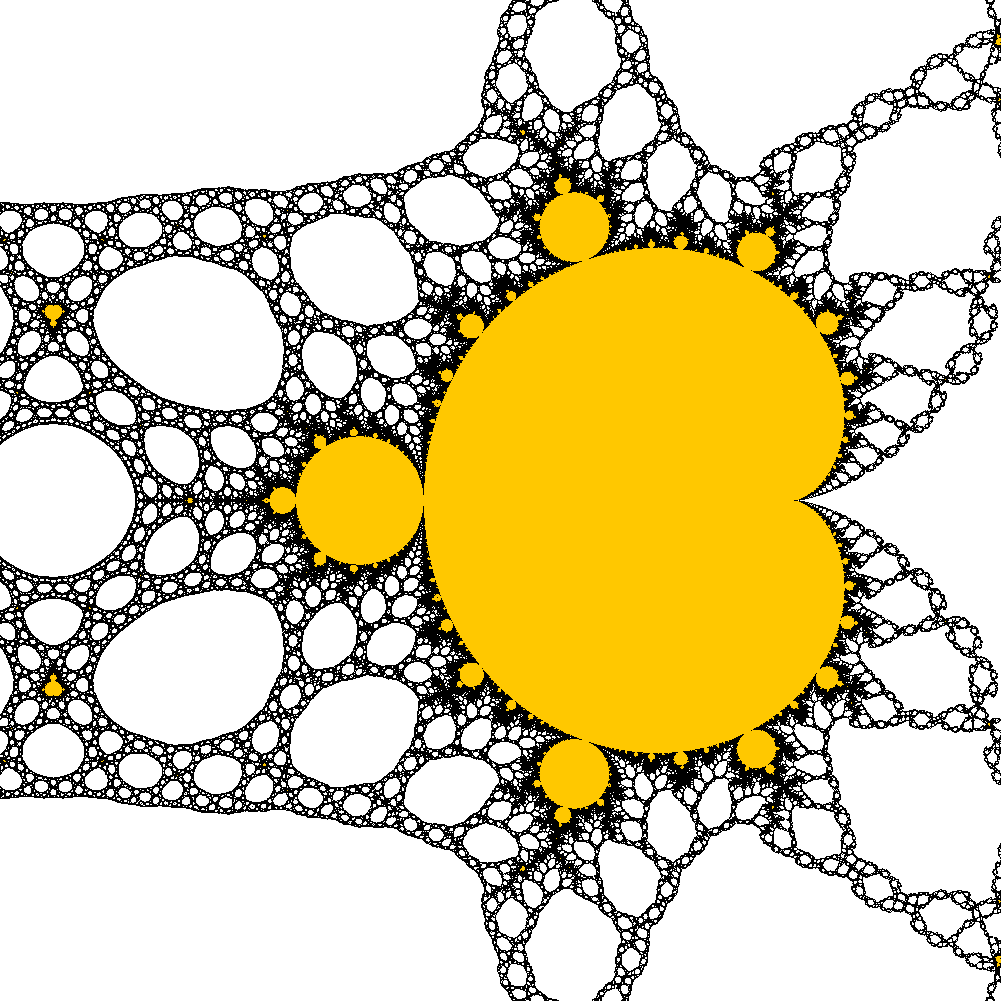}\quad
  \includegraphics[width=0.45\textwidth]{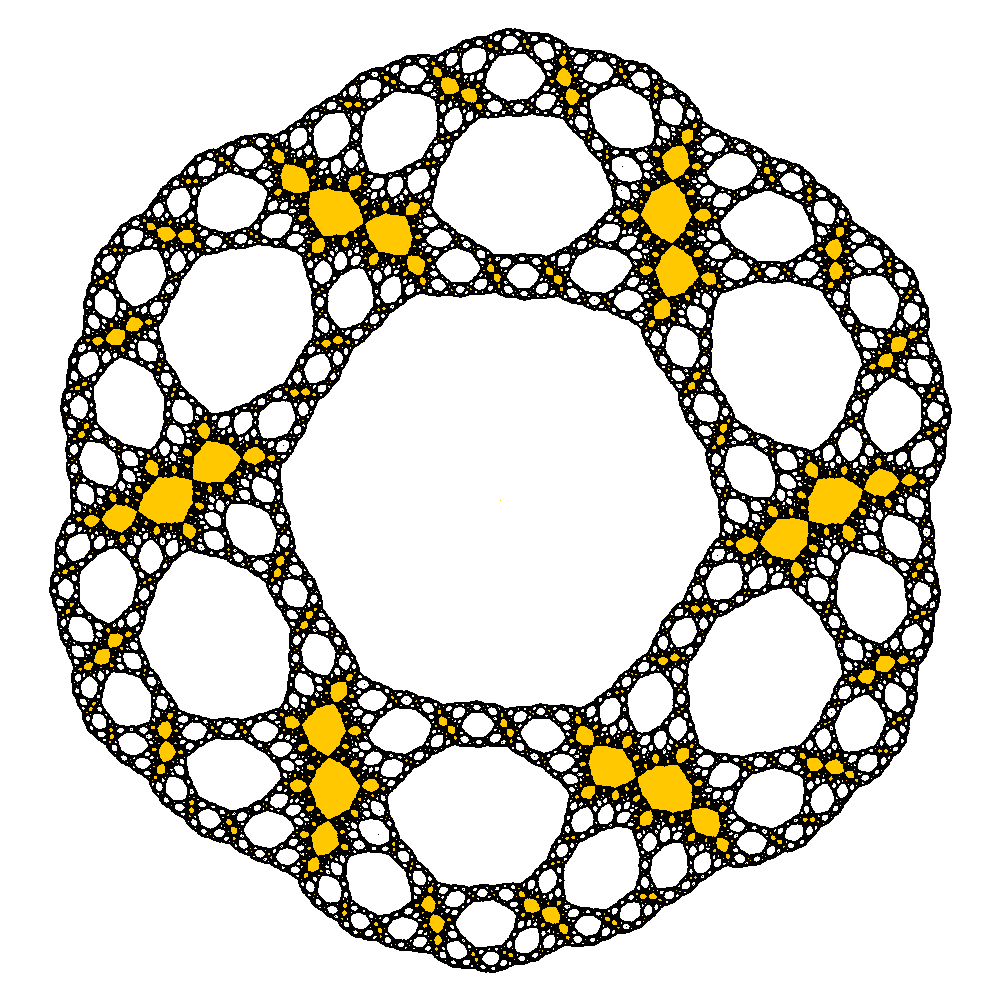}
  \caption{Left: The zoom of the parameter plane of $f_\lambda(z)=z^3+\lambda/z^3$ near renormalizable parameters and a copy $\MM$ of the Mandelbrot set can be seen clearly. Right: The Julia set of $f_{\lambda_0}$, where $\lambda_0$ is chosen such that $f_{\lambda_0}$ is renormlizable and contains some copies of the Julia set of $P_\alpha(z)= e^{2\pi\ii\alpha}z+z^2$ with $\alpha\in\R\setminus\Q$. }
  \label{Fig:McM-parameter}
\end{figure}

In order to find the parameters such that the Julia sets of $f_\lambda$'s have positive area and full Hausdorff dimesion, we need to find some suitable renormalizable parameters in the non-escape locus $\mathcal{N}$. If $\lambda$ is an $1$-renormalizable parameter, recall that $\beta_{f_\lambda}$ and $\beta_{f_\lambda}'$, respectively, are the $\beta$-fixed point and its preimage of the quadratic-like mapping $(f_\lambda, U, V)$ for some Jordan domains $U$ and $V$.

\begin{lema}[{\cite[\S 7]{St06}, \cite[\S 5]{QWY12}, see Figure \ref{Fig:McM-parameter}}]\label{lem:copy}
For the McMullen maps $f_\lambda$ defined in \eqref{equ:McM-2}, we have
\begin{enumerate}
\item There exists a homeomorphic copy $\MM$ of the Mandelbrot set $M$ in the non-escape locus $\mathcal{N}$ such that for each $c\in M$, there exist a unique $\lambda\in\MM$, two Jordan domains $U_\lambda\Subset V_\lambda$ such that $\widetilde{f}_\lambda=(f_\lambda,U_\lambda,V_\lambda)$ is $1$-renormalizable and hybird equivalent to $Q_c$;
\item For all $\lambda\in\MM$, $K(\widetilde{f}_\lambda)\cap\partial B_\lambda=\{\beta_{\tilde{f}_\lambda}\}$ and $K(\widetilde{f}_\lambda)\cap\partial T_\lambda=\{\beta_{\tilde{f}_\lambda}'\}$.
\end{enumerate}
\end{lema}

The statement in Lemma \ref{lem:copy}(b) is crucial in the proof that $f_\lambda$ is a Sierpi\'{n}ski carpet since we need to control the position of the post-critical set of $f_\lambda$.

\subsection{Topology and geometry of quadratic Julia sets}

In this subsection, we state some useful results on quadratic polynomials, including the topological and geometric properties. We first recall a result due to Lyubich and Shishikura.

\begin{thm}[{\cite{Ly91}, \cite{Sh98}}]\label{thm-non-renorm}
There exist non-renormalizable quadratic Julia sets with Hausdorff dimension two and with zero area\footnote{Lyubich proved that the Julia set of a quadratic polynomial has zero area if it has no irrational indifferent periodic points and is not infinitely renormalizable.}.
\end{thm}

One may refer to \cite{YY18} for a more precise construction of the quadratic polynomials in Theorem \ref{thm-non-renorm}. The precise definition of Hausdorff dimension will be given in \S\ref{1-dimension} since in that section we need to use the specific definition to calculate the Hausdorff dimension of some Sierpi\'{n}ski carpets.

\vskip0.1cm
For $\alpha\in\R\setminus\Q$, we define
\begin{equation*}
P_\alpha(z):=e^{2\pi\ii\alpha}z+z^2.
\end{equation*}
The origin is called a \textit{Cremer fixed point} if $P_\alpha$ cannot be locally linearized near the origin. Otherwise, $P_\lambda$ has a simply connected Fatou component containing the origin, which is called the \textit{Siegel disk} of $P_\alpha$ centered at the origin\footnote{Obviously, Cremer fixed points and Siegel disks can be defined similarly for rational maps.}.

\begin{thm}[{\cite{BC12}, \cite{AL15}}]\label{thm:positive-area}
There exist quadratic polynomials that have Julia sets with positive area. Such quadratic polynomials either contain a Cremer fixed point, a Siegel disk or are infinitely renormalizable\footnote{The infinitely renormalizable quadratic polynomials constructed by Buff and Ch\'{e}ritat are infinitely satellite renormalizable while Avila and Lyubich's examples are infinitely primitive renormalizable.}.
\end{thm}

Any irrational number $\alpha\in\R\setminus\Q$ has a continued fractional expansion $\alpha=[a_0;a_1,a_2,\cdots]$, where $a_0\in\Z$ and $a_n\geq 1$ for all $n\geq 1$. Let $N\geq 1$ be a given integer. We denote
\begin{equation}\label{equ:high-type}
\HT_{N}=\{ \alpha=[a_0;a_1,a_2,\cdots]: \ \forall\, n\ge1,\, a_n\ge N \}.
\end{equation}
Every irrational number in $\HT_N$ is called of \textit{high type} (of $N$).

\begin{thm}\label{thm:hairy-cycle}
There exists $N_0\geq 1$ such that for all $\alpha \in \HT_{N_0}$,
\begin{enumerate}
\item (\cite{IS08}) The $\beta$-fixed point of $P_\alpha$ is disjoint with the post-critical set of $P_\alpha$;
\item (\cite{Ch17}, \cite{SY18}) The boundary of the Siegel disk $\Delta_\alpha$ (if any) of $P_\alpha$ is a Jordan curve; and
\item (\cite{Ch17}) If $\partial\Delta_\alpha$ does not contain the critical point (such $\alpha$ exists), then the post-critical set of $P_\alpha$ is a \emph{one-side hairy circle}. In particular, the post-critical set satisfies $P(P_\alpha)\cap (\bigcup_{k\geq 1}P_\alpha^{-k}(\overline{\Delta}_\alpha)\setminus\overline{\Delta}_\alpha)=\emptyset$.
\end{enumerate}
\end{thm}

For the precise definition of one-side hairy circles, one may refer to \cite{Ch17}. Roughly speaking, each one-side hairy circle is a compact set consisting of a Jordan curve on which it attaches uncountably many Jordan arcs, and each Jordan arc is accumulated by at least two sequences of Jordan arcs from both sides.

\section{Carpet Julia sets with positive area and full dimension}

In this section, we will prove Theorems \ref{thm-area} and \ref{thm-area-0-dim-2}. In order to prove that the Julia set $J(f)$ of a rational map $f$ is a Sierpi\'{n}ski carpet, according to \cite{Wh58} one needs to verify the following five conditions:
\begin{itemize}
\item $J(f)$ is compact;
\item $J(f)$ is connected;
\item $J(f)$ is locally connected;
\item $J(f)$ has empty interior; and
\item The Fatou components of $f$ are bounded by pairwise disjoint simple closed curves.
\end{itemize}

Note that the Julia sets of all rational maps are compact and have empty interior provided they are not the whole Riemann sphere. In this section we consider the McMullen maps $f_\lambda$ defined in \eqref{equ:McM-2} with $\lambda\in\MM$. In this case, the Julia set of $f_\lambda$ is not the whole Riemann sphere (since $\infty$ is a super-attracting fixed point), always connected (see Theorem \ref{thm:Escape-3}) and $f_\lambda$ is renormalizable (see Lemma \ref{lem:copy}). Therefore, in order to show that $J(f_\lambda)$ is a Sierpi\'{n}ski carpet for some $\lambda\in\MM$, it is sufficient to prove that $J(f_\lambda)$ is locally connected and that all the Fatou components of $f_\lambda$ are bounded by pairwise disjoint simple closed curves.

\vskip0.1cm
For the local connectivity of connected compact sets in $\EC$, the following criterion is very useful.

\begin{lema}[{Criterion of local connectivity, \cite[Theorem 4.4, pp.\,112-113]{Wh42}}]\label{loc}
A connected compact set $X\subset \EC$ is locally connected if and only if it satisfies the following two conditions:
\begin{itemize}
 \item The boundary of each component of $\EC \setminus X$ is locally connected; and
 \item For any $\varepsilon>0$, there are only a finite number of components of $\EC\setminus X$ with spherical diameter larger than $\varepsilon$.
\end{itemize}
\end{lema}

We will use the following lemma to control the diameters of Fatou components.

\begin{lema}[{Shrinking lemma, \cite{TY96}, \cite{LM97}}]\label{shrinking lemma}
Let $f:\EC \to \EC$ be a rational map and $D$ a topological disk whose closure $\overline{D}$ has no intersection with the post-critical set $P(f)$. Then, either $\overline{D}$ is contained in a Siegel disk or a Herman ring or for any $\varepsilon>0$ there are at most finitely many iterated preimages of $D$ with spherical diameter larger than $\varepsilon$.
\end{lema}

Based on the preceding preparations, now we can give the proof of Theorem \ref{thm-area}.

\begin{proof}[{Proof of Theorem \ref{thm-area}}]
We will give a detailed proof of the existence of $f_\lambda$ such that
\begin{itemize}
\item $J(f_\lambda)$ is a Sierpi\'{n}ski carpet with positive area; and
\item $f_\lambda$ contains a Siegel disk.
\end{itemize}
The proof of the existence of $f_\lambda$ such that $J(f_\lambda)$ is a Sierpi\'{n}ski carpet with positive area and $f_\lambda$ has a Cremer fixed point or is infinitely renormalizable will be only given a sketch.

\vskip0.1cm
\textbf{Step 1.} \textit{The renormalizable Siegel parameters.} Recall that $\MM$ is a copy of the Mandelbrot set embedded in the non-escape locus of $f_\lambda$ (see Lemma \ref{lem:copy}). There exists a parameter $\lambda_1\in\MM$ such that
\begin{itemize}
\item (by Lemma \ref{lem:copy}(a)) There exist two Jordan domains $U_1$ and $V_1$ avoiding the origin such that $\widetilde{f}_1:=(f_1,U_1,V_1)$ is a quadratic-like mapping and contains a fixed Siegel disk $\Delta$, where $f_1:=f_{\lambda_1}$;
\item (by Theorem \ref{thm:hairy-cycle}(b)(c)) The rotation number of $\widetilde{f}_1$ in $\Delta$ is of some high type ($\geq N_0$) and the boundary $\partial\Delta$ is a Jordan curve avoiding any critical point of $f_1$;
\item (by Theorem \ref{thm:hairy-cycle}(c)) The post-critical set of $\widetilde{f}_1$ is a one-side hairy circle; and
\item (by Theorem \ref{thm:positive-area}) The small Julia set of $\widetilde{f}_1$ has positive area\footnote{Buff and Ch\'{e}ritat's construction of quadratic Julia sets with positive area requires that the rotation number is of high type. Moreover, the boundary of the Siegel disk of their construction does not contain the critical point. Actually, it is not known if a quadratic Julia set can have positive area and also contain the boundary of a Siegel disk passing through the critical point.}.
\end{itemize}

For any set $X$ in $\EC$, we denote by $-X:=\{z\in\EC:-z\in X\}$. Without loss of generality, we assume that $n\geq 3$ is odd in \eqref{equ:McM-2}. Then it is easy to see that $(f_1,-U_1,-V_1)$ is also a quadratic-like mapping containing a fixed Siegel disk $-\Delta$. In particular, the above four statements are still true for $(f_1,-U_1,-V_1)$. By the dynamical symmetry of $f_1$, the post-critical set of $f_1$ is
\begin{equation}\label{equ:two-hairy-circle}
P(f_1)=P(\widetilde{f}_1)\cup (-P(\widetilde{f}_1))\cup\{\infty\}
\end{equation}
and every Fatou component of $f_1$ is iterated eventually onto $\Delta$, $-\Delta$ or $B_{\lambda_1}$, where $B_{\lambda_1}$ is the immediate super-attracting basin of $\infty$. Therefore, the Fatou set of $f_1$ is
\begin{equation}\label{equ:Fatou}
F(f_1)=\bigcup_{k\in\N}f_1^{-k}(B_{\lambda_1}\cup\Delta\cup(-\Delta)).
\end{equation}

\textbf{Step 2.} \textit{$J(f_1)$ is locally connected}. Since $\lambda_1\in\MM$, the free critical orbits of $f_1$ are bounded. In particular, they are contained in the Julia set of $f_1$ since each point in $\partial\Delta$ is contained in the closure of some critical point. This means that the Fatou set of $f_1$ does not contain any free critical orbit. According to Theorem \ref{thm:QWY}, all the components of the preimages of $B_{\lambda_1}$ are Jordan domains. On the other hand, all the components of the preimages of $\Delta$ and $-\Delta$ are Jordan domains since $\Delta$ itself is (by the choice of $\lambda_1$ in Step 1). Therefore, all Fatou components of $f_1$ are Jordan domains and the boundary of each Fatou component of $f_1$ is locally connected.

By Lemma \ref{lem:copy}(b) and Theorem \ref{thm:hairy-cycle}(a), the post-critical set of $f_1$ has empty intersection with $\overline{T}_{\lambda_1}$, where $T_{\lambda_1}$ is the Fatou component of $f_1$ containing the origin. Note that $\overline{T}_{\lambda_1}$ is not contained in any Siegel disk or Herman ring. According to Lemma \ref{shrinking lemma}, for any $\varepsilon>0$ there are at most finitely many iterated preimages of $T_{\lambda_1}$ (and hence $B_{\lambda_1}$) with spherical diameter larger than $\varepsilon$. On the other hand, by \eqref{equ:two-hairy-circle}, the post-critical set of $f_1$ in $\C$ is the union of two symmetric hairy circles. Let $\Delta'$ be any connected component of $f_1^{-1}(\Delta\cup (-\Delta))$ which is different from $\Delta$ and $-\Delta$. By Theorem \ref{thm:hairy-cycle}(c), we have $P(f_1)\cap \overline{\Delta'}=\emptyset$. Note that $\overline{\Delta'}$ is not contained in any Siegel disk or Herman ring. By Lemma \ref{shrinking lemma}, for any $\varepsilon>0$ there are at most finitely many iterated preimages of $\Delta$ and $-\Delta$ with spherical diameter larger than $\varepsilon$. By \eqref{equ:Fatou}, it follows that for any $\varepsilon>0$ there are at most finitely many Fatou components of $f_1$ with spherical diameter larger than $\varepsilon$.
According to Lemma \ref{loc}, the Julia set of $f_1$ is locally connected.

\vskip0.1cm
\textbf{Step 3.} \textit{$J(f_1)$ is a Sierpi\'{n}ski carpet with positive area}.
Based on Steps 1 and 2, it is sufficient to show that the boundaries of all the Fatou components of $f_1$ are pairwise disjoint.
Note that the boundaries of any two different Fatou components in $\bigcup_{k\in\N}f_1^{-k}(\Delta\cup(-\Delta))$ are disjoint. Otherwise, it would imply that the boundary of the Siegel disk $\Delta$ contains a critical point, which contradicts the property of the parameter chosen in Step 1. On the other hand, we have $\partial B_{\lambda_1}\cap \partial T_{\lambda_1}=\emptyset$ since $P(f_1)\cap \partial B_{\lambda_1}=\emptyset$. This means that the boundaries of any two different Fatou components in $\bigcup_{k\in\N}f_1^{-k}(B_{\lambda_1})$ are disjoint. Finally, by Lemma \ref{lem:copy}(b) and Theorem \ref{thm:hairy-cycle}(a), we have $\partial \Delta\cap\partial B_{\lambda_1}=\emptyset$. Therefore, the boundaries of all the Fatou components of $f_1$ are pairwise disjoint. This means that we have verified the five conditions stated at the beginning of this section and $J(f_1)$ is a Sierpi\'{n}ski carpet having positive area and containing the boundary of a Siegel disk.

\vskip0.1cm
\textbf{Step 4.} \textit{The Cremer and infinitely renormalizable parameters.}
There exists a parameter $\lambda_2\in\MM$ such that
\begin{itemize}
\item (by Lemma \ref{lem:copy}(a)) There exist two Jordan domains $U_2$ and $V_2$ avoiding the origin such that $\widetilde{f}_2:=(f_2,U_2,V_2)$ is a quadratic-like mapping and hybrid equivalent to a quadratic polynomial $Q_{c_2}$, where $f_2:=f_{\lambda_2}$ and $Q_{c_2}$ has a Cremer fixed point or is infinitely renormalizable;
\item (by Lemma \ref{lem:copy}(b) and Theorem \ref{thm:hairy-cycle}(a)) The post-critical set of $\widetilde{f}_2$ is disjoint with $\overline{T}_{\lambda_2}$; and
\item (by Theorem \ref{thm:positive-area}) The small Julia set of $\widetilde{f}_2$ has empty interior and has positive area.
\end{itemize}

By the dynamical symmetry of $f_2$, we have $P(f_2)\cap\overline{T}_{\lambda_2}=\emptyset$. Note that the Fatou set of $f_2$ is
\begin{equation*}
F(f_2)=\bigcup_{k\in\N}f_2^{-k}(B_{\lambda_2}).
\end{equation*}
Since $f_2(T_{\lambda_2})=B_{\lambda_2}$ and $B_{\lambda_2}$ is a Jordan domain (Theorem \ref{thm:QWY}), it follows that all the Fatou components of $f_2$ are Jordan domains. According to the Shrinking lemma (Lemma \ref{shrinking lemma}) and the characterization of local connectivity (Lemma \ref{loc}), it follows that $J(f_2)$ is locally connected (similar to the arguments in Step 2). Finally, similar to the arguments in Step 3, the boundaries of any two different Fatou components in $\bigcup_{k\in\N}f_2^{-k}(B_{\lambda_2})$ are disjoint. This means that $f_2$ has a Sierpi\'{n}ski carpet Julia set with positive area and contains a Cremer fixed point or is infinitely renormalizable.
\end{proof}

We use the following lemma to prove that some Sierpi\'{n}ski carpet Julia sets have zero area but have Hausdorff dimension two.

\begin{lema}[{\cite[\S 3.3]{Mc94b}}]\label{Mc}
If $f$ is a rational map of degree $>1$, then
\begin{itemize}
 \item The Julia set $J(f)$ is equal to the whole Riemann sphere; or
 \item The spherical distance satisfies $d_{\EC}(f^{\circ k}(z),P(f))\to 0$ for \emph{almost} every $z$ in $J(f)$ as $k\to \infty$.
\end{itemize}
\end{lema}

\begin{proof}[{Proof of Theorem \ref{thm-area-0-dim-2}}]
The idea of the proof is similar to that of Theorem \ref{thm-area}.

\vskip0.1cm
\textbf{Step 1.} \textit{The suitable renormalizable parameters.}
There exists a parameter $\lambda_3\in\MM$ such that
\begin{itemize}
\item (by Lemma \ref{lem:copy}(a)) There exist two Jordan domains $U_3$ and $V_3$ avoiding the origin such that $\widetilde{f}_3:=(f_3^{\circ 2},U_3,V_3)$ is a quadratic-like mapping and hybrid equivalent to a quadratic polynomial $Q_{c_3}$, where $f_3:=f_{\lambda_3}$ and $Q_{c_3}$ is non-renormalizable\footnote{For the non-renormalizable quadratic polynomials in Theorem \ref{thm-non-renorm}, the post-critical sets may equal to the whole Julia sets. In order to guarantee that the post-critical set of $\widetilde{f}_3$ is disjoint with the $\beta$-fixed point, we consider the $2$-renormalization $(f_3^{\circ 2},U_3,V_3)$ but not $1$-renormalization.};
\item (by Lemma \ref{lem:copy}(b)) The post-critical set of $\widetilde{f}_3$ is disjoint with $\overline{T}_{\lambda_3}$; and
\item (by Theorem \ref{thm-non-renorm}) The Julia set $\widetilde{J}_3$ of $\widetilde{f}_3$ has empty interior, Hausdorff dimension two and zero area.
\end{itemize}

\textbf{Step 2.} \textit{$J(f_3)$ is a Sierpi\'{n}ski carpet with Hausdorff dimension two.}
By dynamical symmetry of $f_3$, we have $P(f_3)\cap\overline{T}_{\lambda_3}=\emptyset$. Since $B_{\lambda_3}$ is a Jordan domain (by Theorem \ref{thm:QWY}), it follows that all the Fatou components of $f_3$ are Jordan domains. Similarly, according to the Shrinking lemma and the characterization of local connectivity, it follows that $J(f_3)$ is locally connected. Finally, similar to the arguments in the proof of Theorem \ref{thm-area}, the boundaries of any two different Fatou components in $F(f_3)=\bigcup_{k\in\N}f_3^{-k}(B_{\lambda_3})$ are disjoint. This means that $f_3$ has a Sierpi\'{n}ski carpet Julia set with Hausdorff dimension two.

\vskip0.1cm
\textbf{Step 3.} \textit{The area of $J(f_3)$ is zero}. Without loss of generality, we assume that $n\geq 3$ is odd in \eqref{equ:McM-2}. Then $(f_3^{\circ 2},-U_3,-V_3)$ is also a quadratic-like mapping and hybrid equivalent to the quadratic polynomial $Q_{c_3}$. We consider the following two disjoint subsets of $J(f_3)$:
\begin{equation*}
\begin{split}
\mathcal{J}_1:=&\,\{z\in J(f_3):\exists\, k\geq 0 \text{ such that }f_3^{\circ k}(z)\in \widetilde{J}_3\cup (-\widetilde{J}_3)\}; \text{ and}\\
\mathcal{J}_2:=&\,J(f_3)\setminus \mathcal{J}_1.
\end{split}
\end{equation*}
By the choice of renormalization parameter, it follows that $\Area(\mathcal{J}_1)=0$.
Note that the Julia set of $f_3$ is not the whole Riemann sphere. By Lemma \ref{Mc}, there exists a subset $J'$ of $J(f_3)$ such that $\Area(J')=0$ and for all $z\in J(f_3)\setminus J'$ we have $d_{\EC}(f_3^{\circ k}(z),P(f_3))\to 0$ as $k\to \infty$. 

\vskip0.1cm
For any $z\in\mathcal{J}_2$, there are following two possibilities: 
\begin{enumerate}
\item For all $k\in\N$, $f_3^{\circ k}(z)\not\in U_3\setminus\mathcal{J}_1$ and $f_3^{\circ k}(z)\not\in (-U_3)\setminus\mathcal{J}_1$; or 
\item There exists $k\in\N$ such that $f_3^{\circ k}(z)\in U_3\setminus\mathcal{J}_1$ or $f_3^{\circ k}(z)\in (-U_3)\setminus\mathcal{J}_1$.
\end{enumerate}
Note that $P(f_3)\setminus\{\infty\}\subset \widetilde{J}_3\cup (-\widetilde{J}_3)\subset U_3\cup(-U_3)$. For Case (a), we have $z\in J'$ by definition. For Case (b), there exists $k'>k$ such that $f_3^{\circ k'}(z)\in V_3\setminus U_3$ or $f_3^{\circ k'}(z)\in (-V_3)\setminus (-U_3)$. We still have $z\in J'$. This means that $\mathcal{J}_2\subset J'$ and hence we have $\Area(\mathcal{J}_2)=0$. This proves that $\Area(J(f_3))=0$.
\end{proof}

\section{Carpet Julia sets with arbitrary dimensions}

In this section we will give the proof of Theorem \ref{thm-dim}.
From Theorem \ref{thm:Escape-3}(c), for every $\lambda$ in the Sierpi\'{n}ski holes of $\MH_k$ with $k\geq 3$, $J(f_\lambda)$ is a hyperbolic Sierpi\'{n}ski carpet Julia set. We need the following theorem to obtain the lower bound of the Hausdorff dimension of Sierpi\'{n}ski carpet Julia sets.

\begin{thm}[{\cite[Theorem C]{BW15}}] \label{thm:lower-dim}
For each $n\geq 3$, there exists a parameter $\lambda_n$ in the Sierpi\'{n}ski holes such that the Hausdorff dimension of the Julia set of $f_{\lambda_n}$ satisfies
\begin{equation*}
1\le \dim_{H}J(f_{\lambda_n})\le 1+\frac{10}{\log n}.
\end{equation*}
\end{thm}

Note that $f_{\lambda_n}$ is hyperbolic and the Julia set of $f_{\lambda_n}$ is a Sierpi\'{n}ski carpet. The proof of Theorem \ref{thm:lower-dim} is based on Bowen's formula. For more details on the calculation, see \cite[\S 7]{BW15}.

\vskip0.1cm
The following result is very useful in the calculation of the upper bound of the Hausdorff dimension of the Julia sets of the rational maps in some special hyperbolic components.

\begin{thm}[{\cite[Theorem 6.2]{QY18}}]\label{thm:upper-dim}
Let $\MH$ be a hyperbolic component in the space of rational maps of fixed degree $d\geq 2$. If $f_0\in\MH$ has a simply connected periodic Fatou component whose closure is disjoint with any other Fatou components, then
\begin{equation*}
\sup_{f\in\MH}\dim_H (J(f))=2.
\end{equation*}
\end{thm}

The proof of Theorem \ref{thm:upper-dim} is based on Shishikura's result of the Hausdorff dimension of the Julia sets under parabolic bifurcations \cite[Theorem 2]{Sh98}.

\begin{proof}[{Proof of Theorem \ref{thm-dim}}]
For each $n\geq 3$, let $\MS_n$ be the Sierpi\'{n}ski hole in the parameter space of the McMullen maps containing $\lambda_n$, where $\lambda_n$ is the parameter introduced in Theorem \ref{thm:lower-dim}. Note that the McMullen maps $f_\lambda$ defined in \eqref{equ:McM-2} are of degree $2n$. Let $\widehat{\MS}_n$ be the hyperbolic component in the space of rational maps of degree $2n$ containing the Sierpi\'{n}ski hole $\MS_n$. Then for any $f\in\widehat{\MS}_n$, the Julia set of $f$ is a Sierpi\'{n}ski carpet.

According to \cite{Ru82}, $f\mapsto \dim_H (J(f))$ is a real-analytic (and hence continuous) function as $f$ moves in $\widehat{\MS}_n$. By Theorem \ref{thm:upper-dim}, for each $s\in [1+10/\log n, 2)$, there exists a map $f\in\widehat{\MS}_n$ such that $\dim_H (J(f))=s$. Since $n\geq 3$ can be chosen such that it is arbitrarily large, the theorem follows.
\end{proof}

\section{Dimension of Sierpi\'{n}ski carpets and carpet Julia sets}\label{1-dimension}

In this section, we first construct some Sierpi\'{n}ski carpets (not Julia sets) with Hausdorff dimension one and prove Theorem \ref{thm:dim-1}. Then we show that the Sierpi\'{n}ski carpet Julia sets have Hausdorff dimension strictly larger than one except some special cases.

\subsection{Sierpi\'{n}ski carpets with dimension one}\label{subsec:dim-1}

We first recall the precise definition of Hausdorff dimension, which is needed in the calculation.
Let $X\subset\R^\ell$ with $\ell\geq 1$. For $\delta>0$, the collection of sets $\{U_i\}_{i=1}^\infty$ is called a \textit{$\delta$-cover} of $X$ if $X\subset\bigcup_{i=1}^\infty U_i$ and $0\leq \diam(U_i)\leq\delta$. For $s\geq 0$ we define
\begin{equation}\label{equ:H-delta-s}
\MH_\delta^s(X):=\inf\Big\{\sum_{i=1}^\infty \diam(U_i)^s:\{U_i\}_{i=1}^\infty \text{ is a } \delta\text{-cover of } X\Big\}.
\end{equation}
It is easy to see that the following limit exists:
\begin{equation*}
\MH^s(X):=\lim_{\delta\to 0}\MH_\delta^s(X).
\end{equation*}
The number $\MH^s(X)$ is called the $s$-dimensional \textit{Hausdorff measure} of $X$. The \textit{Hausdorff dimension} of $X$ is defined as (taking the supremum of the empty set to be $0$)
\begin{equation*}
\dim_H(X):=\inf\{s\geq 0:\MH^s(X)=0\}=\sup\{s:\MH^s(X)=\infty\}.
\end{equation*}

\begin{proof}[{Proof of Theorem \ref{thm:dim-1}}]
We start with a unit square $F_0=[0,1]\times[0,1]$. Let $k\geq 3$ be a given integer. Firstly, we remove a square with side length $1-\tfrac{2}{k}$ in the center of $F_0$ and obtain a compact set $F_1\subset F_0$. Note that $F_1$ consists of $4k-4$ squares $\{B_1^1,\cdots,B_1^{4k-4}\}$ with side length $1/k$ whose interiors are pairwise disjoint. Let $F_2\subset F_1$ be the compact set obtained by removing a square with side length $\tfrac{1}{k}(1-\tfrac{2}{k^2})$ in the center of each $B_1^i$ with $1\leq i\leq 4k-4$.
Then $F_2$ consists of $(4k-4)(4k^2-4)$ squares $\{B_2^1,\cdots,B_2^{(4k-4)(4k^2-4)}\}$ with side length $\tfrac{1}{k}\cdot\tfrac{1}{k^2}$ whose interiors are pairwise disjoint. Inductively, in the step $m\geq 1$, we obtain a compact set $F_m$ which consists of $b_m$ squares $\{B_m^1,\cdots,B_m^{b_m}\}$ with side length $l_m$ whose interiors are pairwise disjoint, where
\begin{equation*}
b_m=\prod_{i=1}^m(4k^i-4) \text{\quad and\quad} l_m=\prod_{i=1}^m\frac{1}{k^i}.
\end{equation*}
Let $F=\bigcap_{m=0}^\infty F_m$. Then $F$ is compact, connected, locally connected (by Lemma \ref{loc}), has empty interior and the boundaries of the complementary components are pairwise disjoint Jordan curves. Therefore, $F$ is a Sierpi\'{n}ski carpet (see Figure \ref{Fig:carpet-dim-1}).

\begin{figure}[!htpb]
  \setlength{\unitlength}{1mm}
  \centering
  \includegraphics[width=0.3\textwidth]{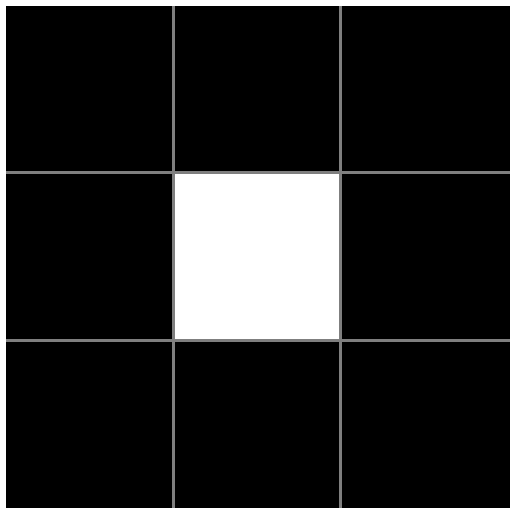}~~
  \includegraphics[width=0.3\textwidth]{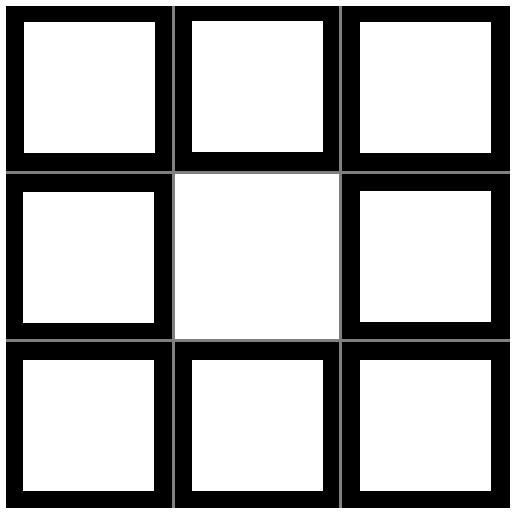}~~
  \includegraphics[width=0.3\textwidth]{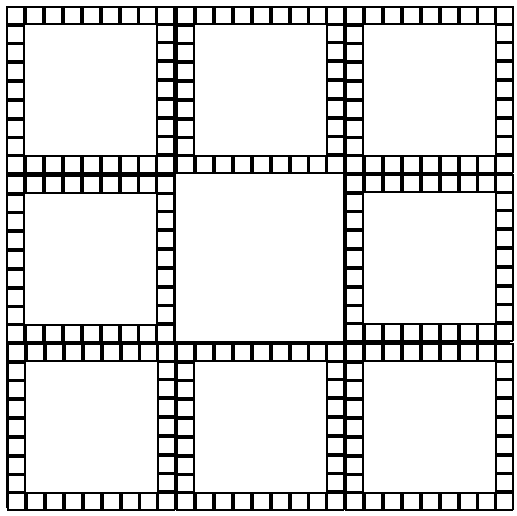}
  \caption{The sketch of constructing a Sierpi\'{n}ski carpet with Hausdorff dimension one. We remove more and more spaces in the squares in the inductive steps (This sketch corresponds to $k=3$ in the proof of Theorem \ref{thm:dim-1}).}
  \label{Fig:carpet-dim-1}
\end{figure}

We now prove $\dim_H(F)=1$ by definition. From the construction of $F$, for any $m\geq 1$, $\{B_m^1,\cdots,B_m^{b_m}\}$ is a $\sqrt{2}\,l_m$-cover of $F$. Let $\varepsilon\in (0,1]$ be any given number. From \eqref{equ:H-delta-s} we have
\begin{equation*}
\MH_{\sqrt{2}\,l_m}^{1+\varepsilon}(F)\leq (\sqrt{2}\,l_m)^{1+\varepsilon}\,b_m=(\sqrt{2})^{1+\varepsilon}\prod_{i=1}^m\frac{4(k^i-1)}{k^{(1+\varepsilon)i}}.
\end{equation*}
Therefore, we have
\begin{equation*}
\MH^{1+\varepsilon}(F)\leq \liminf_{m\to\infty}\,(\sqrt{2}\,l_m)^{1+\varepsilon}\,b_m\leq \liminf_{m\to\infty}\,2\prod_{i=1}^m\frac{4}{k^{\varepsilon i}}=0.
\end{equation*}
This means that $\dim_H(F)\leq 1+\varepsilon$. By the arbitrariness of $\varepsilon$, we have $\dim_H(F)\leq 1$. Since $F$ is a Sierpi\'{n}ski carpet, we have $\dim_H(F)=1$.
\end{proof}

\subsection{Carpet Julia sets with dimension larger than one}\label{subsec:dim-g-1}

We believe that all Sierpi\'{n}ski carpet Julia sets have Haudorff dimension strictly larger than one. Note that the carpets constructed in \S\ref{subsec:dim-1} have no self-similarity. However, Julia sets of rational maps have self-similarity in general. Indeed, the following result holds.

\begin{thm}\label{thm:att-para}
Let $f$ be a rational map having a Sierpi\'{n}ski carpet Julia set. Then the Haudorff dimension of $J(f)$ is strictly larger than one if any one of the following conditions holds:
\begin{enumerate}
\item $f$ contains an attracting basin;
\item $f$ contains a parabolic basin; or
\item $f$ is renormalizable and the small filled Julia set contains a sufficiently high type quadratic Siegel disk.
\end{enumerate}
\end{thm}

In fact, for Theorem \ref{thm:att-para} we will prove that the Haudorff dimension of the boundary of the immediate attracting or parabolic basin is strictly larger than one.
Before giving the proof of Theorem \ref{thm:att-para}, we need the following result.

\begin{lema}\label{lema:renor-attr}
Let $\Omega$ be a fixed attracting or parabolic basin of a rational map $f$. If $\Omega$ is a Jordan domain and $\partial\Omega\cap\partial\Omega'=\emptyset$ for any other Fatou component $\Omega'$,  then there exist an integer $p\geq 1$ and two Jordan domains $U$, $V$ containing $\overline{\Omega}$ such that $(f^{\circ p},U,V)$ is a polynomial-like mapping of degree $d^p\geq 2$, where $d=\deg(f|_{\Omega})$.
\end{lema}

\begin{proof}
Since $\Omega$ is a Jordan domain, there exists a continuous map $\phi:\EC\setminus\Omega\to \EC\setminus\D$ such that $\phi:\EC\setminus\overline{\Omega}\to \EC\setminus\overline{\D}$ is conformal, where $\D$ is the unit disk. Then we obtain a continuous map
\begin{equation*}
g:=\phi\circ f\circ\phi^{-1}:\phi(\EC\setminus f^{-1}(\Omega))\to\EC\setminus\D,
\end{equation*}
such that $g$ is analytic in $\phi(\EC\setminus f^{-1}(\overline{\Omega}))\subset\EC\setminus\overline{\D}$ and $g(\partial\D)=\partial\D$. Since $\partial\Omega\cap\partial\Omega'=\emptyset$ for any other Fatou component $\Omega'$ of $f$, it follows that $\partial\Omega$ does not contain any critical points of $f$, $\phi(\EC\setminus f^{-1}(\overline{\Omega}))$ contains an annulus
\begin{equation*}
A_\delta:=\{z\in\C:1<|z|<1+\delta\} \text{ for some } \delta>0,
\end{equation*}
and $g:\partial\D\to\partial\D$ is a covering map of degree $d$.

By Schwarz reflection principle, $g:\phi(\EC\setminus f^{-1}(\overline{\Omega}))\to\EC\setminus\overline{\D}$ can be extended to a holomorphic map $h:W\to\EC$, where $W$ contains an open neighborhood of $\partial\D$. It is easy to see that $h$ cannot contain any irrational indifferent periodic point on $\partial\D$ since $h(\partial\D)=\partial\D$. Suppose that $h$ has an attracting or parabolic periodic point $z_0\in\partial\D$ with period $l\geq 1$. Let $B$ be the immediate attracting or parabolic basin of $z_0$ in $W$. Then
\begin{equation*}
f^{\circ l}(\phi^{-1}(B\cap(\EC\setminus\overline{\D})))\subset\phi^{-1}(B\cap(\EC\setminus\overline{\D}))
\end{equation*}
and hence $\phi^{-1}(B\cap(\EC\setminus\overline{\D}))$ is contained in a Fatou component $B'$ of $f$ in $\EC\setminus\overline{\Omega}$. Moreover, we have $\phi^{-1}(z_0)\in\partial B'\cap\partial\Omega$, which contradicts the assumption that $\partial\Omega\cap\partial\Omega'=\emptyset$ for any Fatou component $\Omega'\neq \Omega$. Hence, all the periodic points of $h$ on $\partial\D$ are repelling.

By \cite[Theorem A]{Ma85} (see also \cite{Ma87}), it follows that $\partial\D$ is a hyperbolic set, i.e., there exist constants $K>0$ and $\lambda>1$ such that $|(h^{\circ n})'(z)|\geq K\lambda^n$ for all $z\in\partial\D$ and all $n\in\N$. Then there exist an integer $p\geq 1$ and an annulus $A_\varepsilon$ with $\varepsilon>0$ such that $\overline{A}_\varepsilon$ is compactly contained in $h^{\circ p}(A_\varepsilon)$ and $\overline{A}_\varepsilon$ does not contain any critical points of $h$. Let $A_\varepsilon'$ be the connected component of $(h^{\circ p})^{-1}(A_\varepsilon)$ whose boundary contains the unit circle. Then $A_\varepsilon'$ is an annulus which is compactly contained in $A_\varepsilon$ and we obtain a proper holomorphic surjection $h^{\circ p}:A_\varepsilon'\to A_\varepsilon$ of degree $d^p$. Define $U:=\overline{\Omega}\cup\phi^{-1}(A_\varepsilon')$ and $V:=\overline{\Omega}\cup\phi^{-1}(A_\varepsilon)$. Then $(f^{\circ p},U,V)$ is a polynomial-like mapping of degree $d^p$.
\end{proof}

\begin{proof}[{Proof of Theorem \ref{thm:att-para}}]
(a) If $\Omega$ is an attracting basin, then we have $\dim_H(\partial\Omega)>1$ by \cite[Theorem A]{Pr06}.

\vskip0.1cm
(b) Without loss of generality, we assume that $\Omega$ is a fixed parabolic basin of $f$ whose boundary contains a parabolic fixed point $z_0$. Since $J(f)$ is a Sierpi\'{n}ski carpet, it follows that $f'(z_0)=1$ and $\partial \Omega$ is a Jordan curve which is disjoint with the boundaries of any other Fatou components. By Lemma \ref{lema:renor-attr}, there exist an integer $p\geq 1$ and two Jordan domains $U$ and $V$ containing $\overline{\Omega}$ such that $(f^{\circ p},U,V)$ is a polynomial-like mapping. Therefore, we have
\begin{equation*}
f^{\circ p}(U\cap\Omega)\subset\Omega \text{\quad and\quad} f^{\circ p}(U\cap(\EC\setminus\overline{\Omega}))=V\setminus\overline{\Omega}\subset\EC\setminus\overline{\Omega}.
\end{equation*}
This means that $f^{\circ p}$ satisfies condition (f) in \cite[p.\,168]{Ur91}. Since $z_0$ is the unique parabolic fixed point of $f^{\circ p}$ on the Jordan curve $\partial\Omega$, it is easy to see that $f^{\circ p}$ satisfies the conditions (a)-(e) in \cite[p.\,167]{Ur91}. By \cite[Theorem 5.2]{Ur91}, $\partial\Omega$ is either a real-analytic curve or $\dim_H(\partial\Omega)>1$. The former cannot happen since $\partial\Omega$ is quasiconformally homeomorphic to the Julia set of $Q_{1/4}(z)=z^2+1/4$, which is the ``cauliflower" containing infinitely many cusps. Hence $\dim_H(\partial\Omega)>1$. Therefore, in this case we have $\dim_H(J(f))=\dim_H(J(f^{\circ p}))\geq \dim_H(\partial\Omega)>1$.

\vskip0.1cm
(c) Recall that $\HT_N$ is the high type numbers defined in \eqref{equ:high-type}. Since $J(f)$ is a Sierpi\'{n}ski carpet, the boundaries of all the Fatou components are pairwise disjoint. In particular, the boundary of the renormalized high type quadratic Siegel disk $\Delta$ does not contain any critical point. According to \cite{Ch17} and \cite{SY18}, this means that the rotation number $\alpha$ of $\Delta$ does not belong to the Herman type $\mathscr{H}$ provided $\alpha$ is of sufficiently high type. By \cite{CDY18}, there exists a number $N_1\geq 1$ such that for all $\alpha\in\HT_{N_1}\setminus\mathscr{H}$, the post-critical set of $P_\alpha(z)=e^{2\pi\ii\alpha}z+z^2$ has Hausdorff dimension two. This means that $\dim_H(J(f))=2$.
\end{proof}

If the Julia set of rational map $f$ is a Sierpi\'{n}ski carpet and all the periodic Fatou components of $f$ are Siegel disks, then the situation is complicated since we don't have a good control on the post-critical set in general. One cannot expect to obtain the similar result as Theorem \ref{thm:att-para}. On the one hand, the boundary of the Siegel disk $\Delta$ of $f$ could be smooth (see \cite{ABC04}, \cite{BC07}) and $\dim_H(\partial\Delta)=1$. On the other hand, the map $f$ may be non-renormalizable and we cannot use the result of the Hausdorff dimension of the Julia sets of polynomials and polynomial-like mappings (see \cite{Zd90}, \cite{UZ02}). For example, $J(f)$ might be the mating of two quadratic Siegel disks whose boundaries containing no critical points. For the estimation of the Hausdorff dimension of the boundaries of Siegel disks, one may refer to \cite{Mc98} and \cite{GJ02}. In particular, if $\alpha$ is of bounded type, then $1<\dim_H(\Delta_\alpha)<2$, where $\Delta_\alpha$ is the Siegel disk of $P_\alpha(z)=e^{2\pi\ii\alpha}z+z^2$.

\vskip0.1cm
Anyway, Theorem \ref{thm:att-para} gives an evidence that all Sierpi\'{n}ski carpet Julia sets should have Hausdorff dimension strictly larger than one. In particular, if a Sierpi\'{n}ski carpet contains the boundary of a Siegel disk, then the whole Julia set probably has full Hausdorff dimension.

\end{document}